\documentclass[12pt,reqno]{amsart}

\usepackage[arrow,matrix,curve]{xy}

\usepackage[dvips]{graphicx}

\usepackage{amssymb, latexsym, amsmath, amscd, array,
}

\newtheorem{theorem}{Theorem}[section]

\theoremstyle{definition}
\newtheorem{definition}[theorem]{Definition}

\theoremstyle{remark}
\newtheorem{remark}[theorem]{Remark}

\numberwithin{equation}{section}

\newcommand{\E}{{\ensuremath{\mathbb{E}}}}
\newcommand{\zz}{{\ensuremath{\mathbb{Z}}}}
\newcommand\N {{\mathbb N}}
\newcommand\A {{\mathcal A}}

\newcommand\R {{\mathbb R}}
\newcommand\Q {{\mathbb Q}}
\newcommand\Z {{\mathbb Z}}

\newcommand\NNN{\mbox{I\!I\!I\!\!N}}

\newcommand\RRR{\text{\rm I\!I\!R}}

\begin{document}

\thispagestyle{empty}

\title[An integer construction of infinitesimals] {An integer
construction of infinitesimals: Toward a theory of Eudoxus hyperreals
}

\author{Alexandre Borovik} \address{School of Mathematics, University
of Manchester, Oxford Street, Manchester, M13 9PL, United Kingdom}
\email{alexandre.borovik@manchester.ac.uk}

\author{Renling Jin} \address{Department of Mathematics, College of
Charleston, South Carolina 29424, USA}
\email{JinR@cofc.edu}

\author{Mikhail G. Katz$^{0}$} \address{Department of Mathematics, Bar
Ilan University, Ramat Gan 52900 Israel}\email{katzmik@macs.biu.ac.il}

\footnotetext{Supported by the Israel Science Foundation grant
1517/12}



\begin{abstract}
A construction of the real number system based on almost homomorphisms
of the integers~$\Z$ was proposed by Schanuel, Arthan, and others.  We
combine such a construction with the ultrapower or limit ultrapower
construction, to construct the hyperreals out of integers.  In fact,
any hyperreal field, whose universe is a set, can be obtained by such
a one-step construction directly out of integers.  Even the maximal
(i.e., \emph{On}-saturated) hyperreal number system described by
Kanovei and Reeken (2004) and independently by Ehrlich (2012) can be
obtained in this fashion, albeit not in \emph{NBG}.  In \emph{NBG}, it
can be obtained via a one-step construction by means of a definable
ultrapower (modulo a suitable definable class ultrafilter).
\end{abstract}

\maketitle
\tableofcontents

\section{From Kronecker to Schanuel}

Kronecker famously remarked that, once we have the natural numbers in
hand, ``everything else is the work of man'' (see Weber \cite{Web}).
Does this apply to infinitesimals, as well?

The exposition in this section follows R.~Arthan \cite{Arthan04}.  A
function~$f$ from~$\zz$ to~$\zz$ is said to be an {\em almost
homomorphism} if and only if the function~$d_f$ from~$\zz \times \zz$
to~$\zz$ defined by \[d_f(p, q) =f(p+q) - f(p) - f(q)\] has bounded
(i.e., finite) range, so that for a suitable integer~$C$, we
have~$|d_f(p, q)| < C$ for all~$p, q \in \zz$.  The set denoted
\begin{equation}
\label{11z}
{\Z}\rightarrow {\Z}
\end{equation}
of all functions from~$\Z$ to~$\Z$ becomes an abelian group if we add
and negate functions pointwise:
\[
(f+g)(p) = f(p) + g(p), \qquad (-f)(p) = -f(p).
\]
It is easily checked that if~$f$ and~$g$ are almost homomorphisms then
so are~$f+g$ and~$-f$.  Let~$S$ be the set of all almost homomorphisms
from~$\Z$ to~$\Z$.  Then~$S$ is a subgroup of ${\Z}\rightarrow {\Z}$.
Let us write~$B$ for the set of all functions from~$\Z$ to~$\Z$ whose
range is bounded.  Then~$B$ is a subgroup of~$S$.  The ``Eudoxus
reals'' are defined as follows.%
\footnote{\label{eudoxus}The term ``Eudoxus reals'' has gained some
currency in the literature, see e.g., Arthan~\cite{Arthan04}.
Shenitzer~\cite[p.~45]{Sh} argues that Eudoxus anticipated 19th
century constructions of the real numbers.  The attribution of such
ideas to Eudoxus, based on an interpretation involving Eudoxus,
Euclid, and Book~5 of {\em The Elements\/}, may be historically
questionable.}

\begin{definition}
The abelian group~$\E$ of {\em Eudoxus reals} is the quotient group
$S/B$.
\end{definition}

Elements of~$\E$ are equivalence classes,~$[f]$ say, where~$f$ is an
almost homomorphism from~$\Z$ to~$\Z$, i.e.,~$f$ is a function from
$\Z$ to~$\Z$ such that~$d_f(p, q) = f(p, q) -f(p) - f(q)$ defines a
function from~$\Z \times \Z$ to~$\Z$ whose range is bounded. We have
$[f] = [g]$ if and only if the difference~$f - g$ has bounded range,
i.e., if and only if~$|{f(p) - g(p)}| < C$ for some~$C$ and all~$p$ in
$\Z$.

The addition and additive inverse in~$\E$ are induced by the pointwise
addition and inverse of representative almost homomorphisms:
\[
[f] + [g] = [f + g], \qquad -[f] = [-f]
\]
where~$f+g$ and~$-f$ are defined by
\[
(f+g)(p) = f(p) + g(p)
\]
and \[(-f)(p) = -f(p)\] for all~$p$ in $\Z$.

The group~$\E$ of Eudoxus reals becomes an {\em ordered\/} abelian
group if we take the set~$P\subset \E$ of positive elements to be
\[
P = \left\{ [f] \in \E : \sup_{m\in\N\subset\Z} f(m) = +\infty
\right\}.
\]

The multiplication on~$\E$ is induced by {\em composition\/} of almost
homomorphisms.  The multiplication turns~$\E$ into a commutative ring
with unit.  Moreover, this ring is a field.  Even more
surprisingly,~$\E$ is an ordered field with respect to the ordering
defined by $P$.

\begin{theorem}[See Arthan \cite{Arthan04}]
$\E$ is a complete ordered field and is therefore isomorphic to the
field of real numbers~$\R$.
\end{theorem}

The isomorphism~$\R\to \E$ assigns to every real number~$\alpha\in
\R$, the class $[f_\alpha]$ of the function
\begin{eqnarray*}
f_\alpha: \Z & \longrightarrow & \Z \\
n & \mapsto & \lfloor\alpha n \rfloor,
\end{eqnarray*}
where~$\lfloor \,\cdot\, \rfloor$ is the integer part function.

In the remainder of the paper, we combine the above one-step
construction of the reals with the ultrapower or limit ultrapower
construction, to obtain hyperreal number systems directly out of the
integers. We show that any hyperreal field, whose universe is a set,
can be so obtained by such a one-step construction. Following this,
working in \emph{NBG} (von Neumann-Bernays-G\"odel set theory with the
Axiom of Global Choice), we further observe that by using a suitable
\emph{definable ultrapower}, even the maximal (i.e., the
\emph{On-saturated})%
\footnote{Recall that a model $M$ is \emph{On-saturated} if $M$ is
$\kappa$-saturated for any cardinal $\kappa$ in \emph{On}.  Here
\emph{On} (or \emph{ON}) is the class of all ordinals
(cf.~Kunen~\cite[p.~17]{Ku80}).  A hyperreal number system $\langle
\R, {}^*\!\R, S \in \mathfrak{F}\rangle$ is \emph{On-saturated} if it
satisfies the following condition: If $X$ is a set of equations and
inequalities involving real functions, hyperreal constants and
variables, then $X$ has a hyperreal solution whenever every finite
subset of $X$ has a hyperreal solution (see Ehrlich \cite[section 9,
p.~34]{Eh12}).}
hyperreal number system recently described by Ehrlich \cite{Eh12} can
be obtained in a one-step fashion directly from the integers.%
\footnote{Another version of such an \emph{On}-saturated number system
was introduced by Kanovei and Reeken (2004, \cite[Theorem
4.1.10(i)]{KR}) in the framework of axiomatic nonstandard analysis.}
As Ehrlich \cite[Theorem~20]{Eh12} showed, the ordered field
underlying an \emph{On}-saturated hyperreal field is isomorphic to
J. H. Conway's ordered field No, an ordered field Ehrlich describes as
the \emph{absolute arithmetic continuum} (modulo \emph{NBG}).

\section{Passing it through an ultraproduct}

Let now~$\mathcal{Z} = \mathbb{Z} ^\mathbb{N}$ be the ring of integer
sequences with operations of componentwise addition and
multiplication.  We define a {\em rescaling\/} to be a sequence~$\rho
= \langle \rho_n : n \in \mathbb{N}\rangle$ of almost
homomorphisms~$\rho_n: \mathbb{Z} \longrightarrow \mathbb{Z}$.
Rescalings are thought of as acting on~$\mathcal{Z}$, hence the name.
A rescaling~$\rho$ is called bounded if each of its
components,~$\rho_n$, is bounded.

Rescalings factorized modulo bounded rescalings form a commutative
ring~$\mathcal{E}$ with respect to addition and composition.
Quotients of~$\mathcal{E}$ by its maximal ideals are hyperreal fields.
Thus, hyperreal fields are factor fields of the ring of rescalings of
integer sequences.  This description is a tautological translation of
the classical construction, due to E.~Hewitt~\cite{Hew}, but it is
interesting for the sheer economy of the language used.  We will give
further details in the sections below.

 

\section{Cantor, Dedekind, and Schanuel}

The strategy of Cantor's construction of the real numbers%
\footnote{The construction of the real numbers as equivalence classes
of Cauchy sequences of rationals, usually attributed to Cantor, is
actually due to H.~M\'eray (1869, \cite{Me}) who published three years
earlier than E.~Heine and Cantor.}
can be represented schematically by the diagram
\begin{equation}
\label{11}
\R : = \left( \N \to (\Z \times \Z )^{\phantom{I}}_\alpha
\right)_\beta
\end{equation}
where the subscript~$\alpha$ evokes the passage from a pair of
integers to a rational number; the arrow $\to$ alludes to forming
sequences; and subscript~$\beta$ reminds us to select Cauchy sequences
modulo equivalence.  Meanwhile, Dedekind proceeds according to the
scheme
\begin{equation}
\label{22}
\R : = \left( \mathcal{P}(\Z\times\Z)^{\phantom{I}}_\alpha
\right)_\gamma
\end{equation}
where~$\alpha$ is as above,~$\mathcal{P}$ alludes to the set-theoretic
power operation, and~$\gamma$ selects his cuts.  For a history of the
problem, see P.~Ehrlich~\cite{Eh06}.

An alternative approach was proposed by Schanuel, and developed by
N.~A'Campo~\cite{AC}, R.~Arthan~\cite{Arthan01}, \cite{Arthan04},%
\footnote{Arthan's \emph{Irrational construction of $\R$ from $\Z$\/}
\cite{Arthan01} describes a different way of skipping the rationals,
based on the observation that the Dedekind construction can take as
its starting point any Archimedean densely ordered commutative group.
The construction delivers a completion of the group, and one can
define multiplication by analyzing its order-preserving endomorphisms.
Arthan uses the additive group of the ring $\Z[\sqrt{2}]$, which can
be viewed as $\Z\times\Z$ with an ordering defined using a certain
recurrence relation.}
T.~Grundh\"ofer~\cite{Gru}, R.~Street~\cite{St}, O.~Deiser
\cite[pp.~112-127]{De}, and others, who follow the formally simpler
blueprint
\begin{equation}
\label{33}
\R:=(\Z\to\Z)_\sigma
\end{equation}
where~$\sigma$ selects certain almost homomorphisms from~$\Z$ to
itself, such as the map
\begin{equation}
\label{14}
a\mapsto \lfloor ra \rfloor
\end{equation}
for real~$r$, modulo equivalence (think of~$r$ as the ``large-scale
slope'' of the map).%
\footnote{One could also represent a real by a string based on its
decimal expansion, but the addition in such a presentation is highly
nontrivial due to carry-over, which can be arbitrarily long.  In
contrast, the addition of almost homomorphisms is term-by-term.
Multiplication on the reals is induced by composition in~$\Z\to\Z$,
see formula~\eqref{11z}.}
Such a construction has been referred to as the {\em Eudoxus reals\/}.%
\footnote{See footnote~\ref{eudoxus} for a discussion of the term.}
The construction of~$\R$ from~$\Z$ by means of almost homomorphisms
has been described as ``skipping the rationals~$\Q$''.

We will refer to the arrow in~\eqref{33} as the {\em space
dimension\/}, so to distinguish it from the {\em time dimension\/}
occurring in the following construction of an extension of~$\N$:
\begin{equation}
\label{44}
(\N \to \N)_{\tau^{\phantom{I}}_{\!co\!f}}
\end{equation}
where~$\tau^{\phantom{I}}_{\!co\!f}$ identifies sequences
$f,g:\N\to\N$ which {\em differ\/} on a finite set of indices:
\begin{equation}
\label{55}
\{ n \in \N : f(n) = g(n) \} \quad \hbox{is cofinite.}%
\footnote{Note that addition is term-by-term in the time direction, as
well.}
\end{equation}
Here the {\em constant\/} sequences induce an inclusion
\[
\N\to (\N\to\N)_{\tau^{\phantom{I}}_{\!co\!f}}.
\]
Such a construction is closely related to (a version of)
the~$\Omega$-calculus of Schmieden and Laugwitz \cite{SL}.  The
resulting semiring has zero divisors.  To obtain a model which
satisfies the first-order Peano axioms, we need to quotient it
further.  Note that up to this point the construction has not used any
nonconstructive foundational material such as the axiom of choice or
the weaker axiom of the existence of nonprincipal ultrafilters.

\section{Constructing an infinitesimal-enriched continuum}

The traditional ultrapower construction of the hyperreals proceeds
according to the blueprint
\[
(\N\to\R)_{\mathcal U}
\]
where~${\mathcal U}$ is a fixed ultrafilter on~$\N$.  Replacing~$\R$
by any of the possible constructions of~$\R$ from~$\Z$, one in
principle obtains what can be viewed as a direct construction of the
hyperreals out of the integers~$\Z$.  Formally, the most economical
construction of this sort passes via the Eudoxus reals.

An infinitesimal-enriched continuum can be visualized by means of an
infinite-magnification microscope as in Figure~\ref{tamar}.

\begin{figure}
\includegraphics[height=2in]{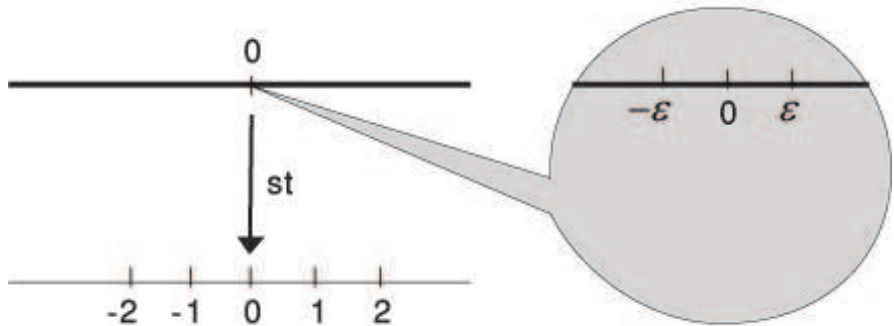}
\caption{Zooming in on infinitesimal~$\epsilon$}
\label{tamar}
\end{figure}

To construct such an infinitesimal-enriched field, we have to deal
with the problem that the semiring
$(\N\to\N)_{\tau^{\phantom{I}}_{\!co\!f}}$ constructed in the previous
section contains zero divisors.

To eliminate the zero divisors, we need to quotient the ring further.
This is done by extending the equivalence relation by means of a
maximal ideal defined in terms of an ultrafilter.  Thus, we extend the
relation defined by~\eqref{55} to the relation declaring~$f$ and~$g$
equivalent if
\begin{equation}
\label{66}
\{ n \in \N : f(n) = g(n) \} \in \mathcal{U},
\end{equation}
where~$ \mathcal{U}$ is a fixed ultrafilter on~$\N$, and add
negatives.  The resulting modification of~\eqref{44}, called an {\em
ultrapower\/}, will be denoted
\begin{equation}
\label{77}
\NNN : = (\N \to \N)_\tau
\end{equation}
and is related to Skolem's construction in 1934 of a nonstandard model
of arithmetic~\cite{Sk}.  We refer to the arrow in \eqref{77} as {\em
time\/} to allude to the fact that a sequence that increases without
bound for large time will generate an infinite ``natural'' number
in~\NNN.  A ``continuous'' version of the ultrapower construction was
exploited by Hewitt in constructing his hyperreal fields in 1948
(see~\cite{Hew}).

The traditional ultrapower approach to constructing the hyperreals is
to start with the field of real numbers~$\R$ and build the ultrapower
\begin{equation}
\label{88}
(\N \to \R)_\tau
\end{equation}
where the subscript~$\tau$ is equivalent to that of~\eqref{77}, see
e.g.~Goldblatt ~\cite{Go}.  For instance, relative to Cantor's
procedure~\eqref{11}, this construction can be represented by the
scheme
\[
(\N \to ( \N \to (\Z \times \Z ) _\alpha )_\beta )_\tau \;;
\]
however, this construction employs needless intermediate procedures as
described above.  Our approach is to follow instead the ``skip the
rationals'' blueprint
\begin{equation}
\label{99}
\RRR:=(\N\to\Z^\N)_{\sigma\tau}
\end{equation}
where the image of each~$a\in \N$ is the sequence~$u^a\in\Z^\N$ with
general term~$u_n^a$, so that~$u^a = \langle u^a_n : n\in \N \rangle$.
Thus a general element of \RRR\; is generated (represented) by the
sequence
\begin{equation}
\label{25}
\left\langle a \mapsto \left( n \mapsto u^a_n \right) : a \in \N
\right\rangle .
\end{equation}
Here one requires that for each fixed element~$n_0\in\N$ of the
exponent copy of~$\N$, the map
\[
a\mapsto u_{n_0}^a
\]
is an almost homomorphism (space direction), while~$\tau$
in~\eqref{99} alludes to the ultrapower quotient in the
time-direction~$n$.  For instance, we can use almost homomorphisms of
type~\eqref{14} with~$r=\frac{1}{n}$.  Then the sequence
\begin{equation}
\label{10}
\left\langle a \mapsto \left( n \mapsto \left\lfloor \tfrac{a}{n}
\right\rfloor \right) : a \in \N \right\rangle
\end{equation}
generates an infinitesimal in~\RRR\; since the almost homomorphisms
are ``getting flatter and flatter'' for large time~$n$.

\begin{theorem}
Relative to the construction~\eqref{99}, we have a natural
inclusion~$\R\subset \hbox{\rm \RRR}$.  Furthermore, \hbox{\rm \RRR}
is isomorphic to the model ${}^*\R$ of the hyperreals obtained by
quotienting~$\R^\N$ by the chosen ultrafilter, as in~\eqref{88}.
\end{theorem}

\begin{proof}
Given a real number~$r\in \R$, we choose the constant sequence given
by~$u_n^a=\lfloor ra\rfloor$ (the sequence is constant in time~$n$).
Sending~$r$ to the element of \RRR\; defined by the sequence
\[
\left\langle a \mapsto \left( n \mapsto \left\lfloor ra \right\rfloor
\right) : a \in \N \right\rangle
\]
yields the required inclusion~$\R\hookrightarrow \RRR$.  The
isomorphism~$\RRR\to \, ^*\R$ is obtained by letting
\[
U_n= \lim \frac{u_n^a}{a}
\]
for each~$n\in\N$, and sending the element of \RRR\; represented
by~\eqref{25} to the hyperreal represented by the sequence~$\langle
U_n : n\in \N \rangle$.
\end{proof}

Denoting by~$\Delta x$ the infinitesimal generated by the integer
object~\eqref{10}, we can then define the derivative of~$y=f(x)$
at~$x$ following Robinson as the real number~$f'(x)$ infinitely close
(or, in Fermat's terminology, {\em adequal\/})%
\footnote{See A.~Weil~\cite[p.~1146]{We}.}
to the infinitesimal ratio~$\frac{\Delta y}{\Delta x}\in\RRR$.

Applications of infinitesimal-enriched continua range from aid in
teaching calculus (\cite{El}, \cite{KK1}, \cite{KK2}, \cite{KT},
\cite{NB}) to the Bolzmann equation (see L.~Arkeryd~\cite{Ar81,
Ar05}); modeling of timed systems in computer science (see H.~Rust
\cite{Rust}); mathematical economics (see R.~Anderson \cite{An00});
mathematical physics (see Albeverio {\em et al.\/} \cite{Alb}); etc.
A comprehensive re-appraisal of the historical antecedents of modern
infinitesimals has been undertaken in recent work by B\l aszczyk et
al.~\cite{BKS}, Borovik et al.~\cite{BK}, Br\aa ting \cite{Br},
Kanovei \cite{Kan}, Katz \& Katz \cite{KK11a, KK11b, KK11c, KK11d},
Katz \& Leichtnam \cite{KL}, Katz and Sherry \cite{KS1, KS2}, and
others.  A construction of infinitesimals by ``splitting'' Cantor's
construction of the reals is presented in Giordano et al.~\cite{GK11}.

\section{Formalization}
\label{five}

In this and the next sections we formalize and generalize the
arguments in the previous sections.  We show that by a one-step
construction from $\Z$-valued functions we can obtain any given
(set) hyperreal field. We can even obtain a universal hyperreal field
which contains an isomorphic copy of every hyperreal field,
by a one-step construction from $\Z$-valued functions.

We assume that the reader is familiar with some basic concepts of
model theory.  Consult \cite{CK} or \cite{Ke} for concepts and
notations undefined here.

Let $\mathfrak{A}$ and $\mathfrak{B}$ be two models in a language
$\mathcal{L}$ with base sets $A$ and~$B$, respectively.  The model
$\mathfrak{B}$ is called an $\mathcal{L}$-elementary extension of the
model $\mathfrak{A}$, or $\mathfrak{A}$ is an $\mathcal{L}$-elementary
submodel of $\mathfrak{B}$, if there is an embedding $e: A\rightarrow
B$, called an $\mathcal{L}$-elementary embedding, such that for any
first-order $\mathcal{L}$-sentence $\varphi(a_1,a_2,\ldots,a_n)$ with
parameters $a_1,a_2,\ldots,a_n\in A$, $\varphi(a_1,a_2,\ldots,a_n)$ is
true in $\mathfrak{A}$ if and only if
$\varphi(e(a_1),e(a_2),\ldots,e(a_n))$ is true in $\mathfrak{B}$.

Let $\mathfrak{A}$ and $\mathfrak{B}$ be two models in language $\mathcal{L}$
with base sets $A$ and $B$, respectively. Let
\[
\mathcal{L}'=\mathcal{L}\cup \{P_R:\exists m\in\N, R\subseteq A^m\},
\]
i.e., $\mathcal{L}'$ is formed by adding to $\mathcal{L}$ an
$m$-dimensional relational symbol $P_R$ for each $m$-dimensional
relation $R$ on $A$ for any positive integer $m$.  Let $\mathfrak{A}'$ be
the natural $\mathcal{L}'$-model with base set $A$, i.e., the
interpretation of $P_R$ in $\mathfrak{A}'$ for each $R\subseteq A^m$ is
$R$. The model $\mathfrak{B}$ is called a {\em complete elementary
extension\/} of $\mathfrak{A}$ if $\mathfrak{B}$ can be expanded to an
$\mathcal{L}'$-model $\mathfrak{B}'$ with base set $B$ such that
$\mathfrak{B}'$ is an $\mathcal{L}'$-elementary extension of $\mathfrak{A}'$.

It is a well-known fact that if $\mathfrak{B}$ is an ultrapower of
$\mathfrak{A}$ or a limit ultrapower of $\mathfrak{A}$,
then $\mathfrak{B}$ is a complete elementary extension of $\mathfrak{A}$.

In this section we always view the set $\R$ as the set of all Eudoxus
reals.

An ordered field is called a hyperreal field if it is a proper
complete elementary extension of $\R$.  Let
\[
\mathcal{L}'=\{+,\cdot,\leqslant,0,1,P_R\}_{R\in\mathcal{R}}
\]
where $\mathcal{R}$ is the collection of all finite-dimensional
relations on $\R$.  We do not distinguish between $\R$ and the
$\mathcal{R}'$-model
$\mathfrak{R}=(\R;+,\cdot,\leqslant,0,1,R)_{R\in\mathcal{R}}$.  By
saying that $^*\R$ is a hyperreal field we will sometimes mean that
$^*\R$ is the base set of the hyperreal field, but at other times we
mean that~$^*\R$ is the hyperreal field viewed as an
$\mathcal{L}'$-model.  We will spell out the distinction when it
becomes necessary.

\medskip

Recall that $S$ is the set of all bounded functions from $\Z$ to~$\Z$.
For a pair of almost homomorphisms $f,g:\Z\rightarrow\Z$, we will
write $f\sim_{\sigma} g$ if and only if $f-g\in S$. Let $I$ be an
infinite set.  If $F(x,y)$ is a two-variable function from $\Z\times
I$ to $\Z$ and $i$ is a fixed element in $I$, we write $F(x,i)$ for
the one-variable function $F_i(x)=F(x,i)$ from $\Z$ to $\Z$.
\begin{definition}
Let $I$ be any infinite set. We set
\begin{eqnarray*}
\lefteqn{\A(\Z\times I,\Z)=}\\ & &=\{F\in\Z^{\Z\times I}:\forall i\in
I,\, F(x,i)\,\mbox{ is an almost homomorphism.}\}
\end{eqnarray*}
\end{definition}
Let $\mathcal{U}$ be a fixed non-principal ultrafilter on $I$. For a
pair of functions $f,g:I\rightarrow J$ for some set $J$, we set
\[
f\sim_{\tau} g\,\mbox{ if and only if }\, \{i\in
I:f(i)=g(i)\}\in\mathcal{U}.\] Let $[f]_{\tau}=\{g\in I^J:g\sim_{\tau}
f\}$.

\begin{definition}
For any $F,G\in \A(\Z\times I,\Z)$ we will write
\[
F\sim_{\sigma\tau}G\,\mbox{ if and only if }\, \{i\in
I:F(x,i)\sim_{\sigma} G(x,i)\}\in\mathcal{U}.
\]
\end{definition}
It is easy to check that $\sim_{\sigma\tau}$ is an equivalence
relation on $\A(\Z\times I,\Z)$. For each $F\in\A(\Z\times I,\Z)$ let
\[
[F]_{\sigma\tau}=\{G\in\A(\Z\times I,\Z):G\sim_{\sigma\tau} F\}.
\]
For each $F(x,y)\in\A(\Z\times I,\Z)$ we can consider
$[F(\cdot,y)]_{\sigma}$ as a function of $y$ from $I$ to $\R$. Thus
the map
\[
\Phi:\A(\Z\times I,\Z)/\!\sim_{\sigma\tau}\rightarrow\R^I/\mathcal{U}
\]
such that $\Phi([F]_{\sigma\tau})=[[F]_{\sigma}]_{\tau}$ is an
isomorphism from $\A(\Z\times I,\Z)/\!\sim_{\sigma\tau}$ to
$\R^I/\mathcal{U}$. Hence $\A(\Z\times I,\Z)/\!\sim_{\sigma\tau}$ can
be viewed as an ultrapower of $\R$. Therefore, the quotient
\[
\RRR_I=\A(\Z\times I,\Z)/\!\sim_{\sigma\tau}
\]
is a hyperreal field constructed in one step from the set of
$\Z$-valued functions $\A(\Z\times I,\Z)$.

If the set $I$ is $\N$, then $\A(\Z\times \N,\Z)/\!\sim_{\sigma\tau}$
is exactly the hyperreal field~$\RRR$ mentioned in the previous
sections.  Since $I$ can be any infinite set, we can construct a
hyperreal field of arbitrarily large cardinality in one step from a
set of $\Z$-valued functions $\A(\Z\times I,\Z)$.

\section{Limit ultrapowers and definable ultrapowers}
\label{six}

If we consider a limit ultrapower instead of an ultrapower, we can
obtain any (set) hyperreal field by a one-step construction from a set
of $\Z$-valued functions $\A(\Z\times I,\Z)|\mathcal{G}$. The reader
could consult (Keisler~\cite{Ke}) for the notations, definitions, and
basic facts about limit ultrapowers.  The main fact we need here is
the following theorem (see Keisler \cite[Theorem~3.7]{Ke}).

\begin{theorem}
If $\mathfrak{A}$ and $\mathfrak{B}$ are two models of the same
language, then~$\mathfrak{B}$ is a complete elementary extension of
$\mathfrak{A}$ if and only if $\mathfrak{B}$ is a limit ultrapower of
$\mathfrak{A}$.
\end{theorem}

Given any (set) hyperreal field $^*\R$, let $\mathcal{U}$ be the
ultrafilter on an infinite set $I$ and let $\mathcal{G}$ be the filter
on $I\times I$ such that $^*\R$ is isomorphic to the limit ultrapower
$(\R^I/\mathcal{U})|\mathcal{G}$. We can describe the limit ultrapower
$(\R^I/\mathcal{U})|\mathcal{G}$ in one step from the set of
$\Z$-valued functions $\A(\Z\times I,\Z)|\mathcal{G}$. For each
$F\in\A(\Z\times I,\Z)$, let
\[
eq(F)=\{(i,j)\in I\times I:[F(x,i)]_{\sigma}=[F(x,j)]_{\sigma}\}.
\]
Let
\[
\A(\Z\times I,\Z)|\mathcal{G}=\{F\in \A(\Z\times
I,\Z):eq(F)\in\mathcal{G}\}.\] Notice that $\A(\Z\times
I,\Z)|\mathcal{G}$ is a subset of $\A(\Z\times I,\Z)$. Hence
$(\A(\Z\times I,\Z)|\mathcal{G})/\!\sim_{\sigma\tau}$ can be viewed as
a subset of $\A(\Z\times I,\Z)/\!\sim_{\sigma\tau}$. Again for each
\[
F\in \A(\Z\times I,\Z)/\!\sim_{\sigma\tau}
\]
let
\[
\Phi([F]_{\sigma\tau})=[[F]_{\sigma}]_{\tau}.
\]
Then $\Phi$ is an isomorphism from
\[
(\A(\Z\times I,\Z)|\mathcal{G})/\!\sim_{\sigma\tau}
\]
to $(\R^I/\mathcal{U})|\mathcal{G}$. Therefore, $(\A(\Z\times
I,\Z)|\mathcal{G})/\!\sim_{\sigma\tau}$ as an elementary subfield of
$\A(\Z\times I,\Z)/\!\sim_{\sigma\tau}$ is isomorphic to the hyperreal
field $^*\R$.

\begin{theorem}
An isomorphic copy of any (set) hyperreal field $^*\R$ can be obtained
by a one-step construction from a set of $\Z$-valued functions
$\A(\Z\times I,\Z)|\mathcal{G}$ for some filter $\mathcal{G}$ on
$I\times I$.
\end{theorem}

\section{Universal and \emph{On}-saturated Hyperreal Number Systems}

We call a hyperreal field $\RRR$ {\it universal} if any hyperreal
field, which is a set or a proper class of \emph{NBG}, can be
elementarily embedded in~$\RRR$.  Obviously a universal hyperreal
field is necessarily a proper class.  We now want to construct a
definable hyperreal field with the property that any definable
hyperreal field that can be obtained in \emph{NBG} by a one-step
construction from a collection of $\Z$-valued functions can be
elementarily embedded in it. In a subsequent remark we point out that
in \emph{NBG} we can actually construct a definable hyperreal field so
that every hyperreal field (definable or non-definable) can be
elementarily embedded in it. Moreover, the universal hyperreal field
so constructed is isomorphic to the \emph{On}-saturated hyperreal
field described in~\cite{Eh12}.

Notice that \emph{NBG} implies that there is a well order $\leqslant_V$
on $V$ where $V$ is the class of all sets. A class $X\subseteq V$ is
called $\Delta_0$-definable if there is a first-order formula
$\varphi(x)$ with set parameters in the language $\{\in,\leqslant\}$
such that for any set $a\in V$, $a\in X$ if and only if $\varphi(a)$
is true in $V$.  Trivially, every set is $\Delta_0$-definable.  We
work within a model of \emph{NBG} with set universe $V$ plus all
$\Delta_0$-definable proper subclasses of $V$.  By saying that a class
$A$ is definable we mean that $A$ is $\Delta_0$-definable.

Let $\Sigma$ be the class of all finite sets of ordinals, i.e.,
$\Sigma=On^{<\omega}$. Notice that $\Sigma$ is a definable proper class.
Let $\mathcal{P}$ be the
collection of all definable subclasses of $\Sigma$.  Notice
that we can code $\mathcal{P}$ by a definable class.%
\footnote{ This is true because each definable subclass of $\Sigma$
can be effectively coded by the G\"{o}del number of a first-order
formula in the language of $\{\in,\leqslant_V\}$ and a set in~$V$.  By
the well order of $V$ we can determine a unique code for each
definable class in~$\mathcal{P}$.}
Using the global choice, we can form a non-principal definable
ultrafilter $\mathcal{F}\subseteq\mathcal{P}$ such that for each
$\alpha\in On$, the definable class
\[
\hat{\alpha}=\{s\in\Sigma:\alpha\in s\}
\]
is in $\mathcal{F}$.  Again $\mathcal{F}$ can be coded by a definable
class.  Let $\A_0(\Z\times\Sigma,\Z)$ be the collection of all
definable class functions $F$ from $\Z\times\Sigma$ to $\Z$ such that
for each $s\in\Sigma$, $F(x,s)$ is an almost homomorphism from $\Z$ to
$\Z$.  For any two functions $F$ and $G$ in $\A_0(\Z\times\Sigma,\Z)$,
we write $F\sim_{\sigma\tau} G$ if and only if the definable class
\[
\{s\in\Sigma:F(x,s)-G(x,s)\in S\}
\]
is in $\mathcal{F}$. Let $[F]_{\sigma\tau}$ be the collection of all
definable classes $G$ in $\A_0(\Z\times\Sigma,\Z)$ such that
$F\sim_{\sigma\tau} G$.  Then $\sim_{\sigma\tau}$ is an equivalence
relation on $\A_0(\Z\times\Sigma,\Z)$.  Let
\[
\RRR_0=\A_0(\Z\times\Sigma,\Z)/\!\sim_{\sigma\tau}.
\]
By the arguments employed before, we can show that $\RRR_0$ is
isomorphic to the definable ultrapower of $\R$ modulo
$\mathcal{F}$. Hence $\RRR_0$ is a complete elementary extension of
$\R$. Therefore, $\RRR_0$ is a hyperreal field. By slightly modifying
the proof of \cite[Theorem 4.3.12, p.~255]{CK} we can prove the
following theorem.

\begin{theorem}\label{maximal}
$\RRR_0$ is a class hyperreal field, and any definable hyperreal field
$^*\R$ admits an elementary imbedding into $\RRR_0$.
\end{theorem}

\begin{proof}
We only need to prove the second part of the theorem.  For notational
convenience we view $\RRR_0$ as $\R^{\Sigma}/\mathcal{F}$ instead of
$\A_0(\Z\times\Sigma,\Z)/\sim_{\sigma\tau}$ in this proof.

Given a definable hyperreal field $^*\R$, recall that $\mathcal{L}$ is
the language of ordered fields, and
\[
\mathcal{L}'=\mathcal{L}\cup\{P_R:R\,\mbox{ is a finite-dimensional
relation on }\,\R\}.
\]
Let $\Lambda_{^*\R}$ be all quantifier-free $\mathcal{L}'$-sentences
$\varphi(r_1,r_2,\ldots,r_m)$ with parameters $r_i\in\,^*\R$ such that
$\varphi(r_1,r_2,\ldots,r_m)$ is true in $^*\R$. Since $^*\R$ is a
definable class, $\Lambda_{^*\R}$ is a definable class (under a proper
coding). Let $\kappa$ be the size of $\Lambda_{^*\R}$, i.e., $\kappa$
is the cardinality of $^*\R$ if $^*\R$ is a set and $\kappa=On$ if
$^*\R$ is a definable proper class.  Let $j$ be a definable bijection
from $\kappa$ to $\Lambda_{^*\R}$.

For each $r\in\,^*\R$ we need to find a definable function
$F_r:\Sigma\rightarrow\R$ such that the map $r\mapsto [F_r]_{\tau}$ is an
$\mathcal{L}'$-elementary embedding. We define these $F_r$ simultaneously.

Let $s\in\Sigma$. If $s\cap\kappa=\emptyset$ let $F_r(s)=0$.  Suppose
$s\cap\kappa\not=\emptyset$ and let $s'=s\cap\kappa$.  Notice that if
$\kappa=On$, then $s=s'$. Let
$\varphi_s(r_1,r_2,\ldots,r_m)=\bigwedge_{\alpha\in s'}j(\alpha)$.
Since \[\exists x_1,x_2,\ldots,x_m\varphi_s(x_1,x_2,\ldots,x_m)\] is
true in $^*\R$, it is also true in $\R$. Let
$(a_1,a_2,\ldots,a_m)\in\R^m$ be the $\leqslant_V$-least $m$-tuple
such that $\varphi_s(a_1,a_2,\ldots,a_m)$ is true in $\R$. If
$r\not\in\{r_1,r_2,\ldots,r_m\}$, let $F_r(s)=0$. If $r=r_i$ for
$i=1,2,\ldots,m$, then let $F_r(s)=a_i$. Since $\varphi_s$ is
quantifier-free, the functions $F_r$ are definable classes in \emph{NBG}.

We now verify that $\Phi:\,^*\R\rightarrow\RRR_0$ such that
$\Phi(r)=[F_r]_{\tau}$ is an $\mathcal{L}'$-elementary embedding.

Let $\phi(r_1,r_2,\ldots,r_m)$ be an arbitrary $\mathcal{L}'$-sentence
with parameters \newline $r_1,r_2,\ldots,r_m\in\,^*\R$.

Suppose that $\phi(r_1,r_2,\ldots,r_m)$ is true in $^*\R$. Since
$\phi(x_1,x_2,\ldots,x_m)$ defines an $m$-ary relation $R_{\phi}$ on
$\R$, we have that the $\mathcal{L}'$-sentence
\[\eta=:\forall x_1,x_2,\ldots,x_m(\phi(x_1,x_2,\ldots,x_m)\leftrightarrow
R_{\phi}(x_1,x_2,\ldots,x_m))\] is true in $\R$. Hence $\eta$ is true
in $^*\R$ and in $\RRR_0$.  One of the consequences of this is that
$R_{\phi}(r_1,r_2,\ldots,r_m)$ is true in $^*\R$, hence it is in
$\Lambda_{^*\R}$.  Let $\alpha\in\kappa$ be such that
$j(\alpha)=R_{\phi}(r_1,r_2,\ldots,r_m)$.  If $\alpha\in s$, then
\[
\varphi_s=R_{\phi}(r_1,r_2,\ldots,r_m)\wedge\bigwedge_{\beta\in
s',\beta\not=\alpha}j(\beta).\] Hence
$R_{\phi}(F_{r_1}(s),F_{r_2}(s),\ldots,F_{r_m}(s))$ is true in $\R$ by
the definition of the $F_r$'s.  Since $\eta$ is true in $\R$, we have
that $\phi(F_{r_1}(s),F_{r_2}(s),\ldots,F_{r_m}(s))$ is true in
$\R$. Thus
\[\{s\in\Sigma:\phi(F_{r_1}(s),F_{r_2}(s),\ldots,F_{r_m}(s)) \,\mbox{
is true in }\,\R\}\supseteq\hat{\alpha}.\] Since $\hat{\alpha}$ is a
member of $\mathcal{F}$, we have that
$\phi([F_{r_1}]_{\tau},[F_{r_2}]_{\tau},\ldots,[F_{r_m}]_{\tau})$ is
true in $\R^{\Sigma}/\mathcal{F}$.

Suppose that $\phi(r_1,r_2,\ldots,r_m)$ is false in $^*\R$.  Then
$\neg\phi(r_1,r_2,\ldots,r_m)$ is true in $^*\R$. Hence by the same
argument we have that
\[
\neg\phi([F_{r_1}]_{\tau},[F_{r_2}]_{\tau},\ldots,[F_{r_m}]_{\tau})\,\mbox{
is true in }\,\R^{\Sigma}/\mathcal{F}\] which implies that
\[\phi([F_{r_1}]_{\tau},[F_{r_2}]_{\tau},\ldots,[F_{r_m}]_{\tau})\,\mbox{
is false in }\,\R^{\Sigma}/\mathcal{F}.\] Hence $\Phi(r)=[F_r]_{\tau}$
is an $\mathcal{L}'$-elementary embedding from $^*\R$ to $\RRR_0$.
\end{proof}

\begin{remark}
We have shown that every definable hyperreal field can be elementarily
embedded into~$\RRR_0$.  If we want to show that $\RRR_0$ is
universal, we need to elementarily embed every (definable or
non-definable) hyperreal field $^*\R$ into the definable hyperreal
field $\RRR_0$. Notice that there are models of \emph{NBG} with
non-definable classes.  The proof of Theorem \ref{maximal} may not
work when $^*\R$ is a non-definable class because the bijection $j$
may not be definable and $F_r$ may not be definable.  If $F_r$ is not
definable, $[F_r]_{\tau}$ may not be an element in $\RRR_0$.

The idea of making every hyperreal field embeddable into $\RRR_0$ is
that we can make $\RRR_0$ \emph{On}-saturated by selecting a definable
ultrafilter $\mathcal{F}$ more carefully. Notice that \emph{NBG}
implies that every proper class has the same size \emph{On}. Hence
$^*\R$ can be expressed as the union of \emph{On}-many sets. If we can
make sure that $\RRR_0$ is \emph{On}-saturated, i.e.,
$\alpha$-saturated for any set cardinality $\alpha$, then $^*\R$ can
be elementarily embedded into~$\RRR_0$ although such an embedding may
be non-definable.

The ultrafilter $\mathcal{F}$ used in the construction of $\RRR_0$ in
the proof of Theorem \ref{maximal} is a definable version of a {\it
regular} ultrafilter.  In order to make sure that $\RRR_0$ is
\emph{On}-saturated, we need to require that $\mathcal{F}$ be a
special kind of definable regular ultrafilter called a definable {\it
good} ultrafilter. The definition of a (set) good ultrafilter can be
found in \cite[p.~386]{CK}.  The construction of an $\alpha^+$-good
ultrafilter can be found in either \cite[Theorem 6.1.4]{CK} or Kunen
\cite{Ku}.  By the same idea of constructing $\mathcal{F}$ above we
can follow the steps in \cite{Ku} or \cite{CK} to construct a
definable class good ultrafilter $\mathcal{F}$ on \emph{On}. Now the
ultrapower $\RRR_0$ of $\R$ modulo the definable class good
ultrafilter $\mathcal{F}$ is \emph{On}-saturated. The proof of this
fact is similar to that in \cite{CK}.  However, since the definition
of a definable class good ultrafilter and the proof of the saturation
property of the ultrapower modulo a definable class good ultrafilter
are long and technical, and the ideas are similar to what we have
already presented above, we will not include them in this paper.
\end{remark}

Another way of constructing a definable \emph{On}-saturated hyperreal
field $\RRR_0$ is by taking the union of an \emph{On}-long definable
elementary chain of set hyperreal fields $\left\{ {}^*\R_\alpha :
\alpha\in On\right\}$ with the property that ${}^*\R_\alpha$ is
$|\alpha|$-saturated.  However, this construction cannot be easily
translated into a ``one-step" construction.  Moreover, if we allow
higher-order classes, we can express $\RRR_0$ as a one-step limit
ultrapower following the same idea as in the proof of \cite[Theorem
6.4.10]{CK}.  However, this is not possible in \emph{NBG} since all
classes allowed in a model of \emph{NBG} are subclasses of $V$.
%
%
On the other hand, as we indicated above, the process of constructing
$\RRR_0$ as a definable ultrapower can be carried out in \emph{NBG},
and done so in a ``one-step" fashion.

\section*{Acknowledgments}

We are grateful to R.~Arthan, P.~Ehrlich, and V.~Kanovei for helpful
comments.

\end{document}